\documentclass[11pt]{article}

\usepackage[T1]{fontenc}
\usepackage{lmodern}
\usepackage[margin=1in]{geometry}
\usepackage{amsmath,amssymb,amsthm,mathtools}
\usepackage{microtype}
\numberwithin{equation}{section}
\usepackage[hidelinks]{hyperref}

\newtheorem{theorem}{Theorem}[section]
\newtheorem{proposition}[theorem]{Proposition}
\newtheorem{lemma}[theorem]{Lemma}
\newtheorem{corollary}[theorem]{Corollary}

\newcommand{\Sphere}{\mathbb S}
\newcommand{\R}{\mathbb R}
\newcommand{\K}{\mathcal K}
\newcommand{\C}{\mathcal C}
\newcommand{\sgn}{\operatorname{sgn}}
\newcommand{\pos}{\operatorname{pos}}
\newcommand{\conv}{\operatorname{conv}}
\newcommand{\convS}{\conv_{\Sphere}}
\newcommand{\indicator}{\mathbf{1}}
\newcommand{\dd}{\,\mathrm d}

\hypersetup{
  pdftitle={The Spherical Hadwiger Theorem},
  pdfauthor={Suijie Wang and Shengguo Wu},
  pdfsubject={Geometric valuations and spherical convexity},
  pdfkeywords={geometric valuations, spherical convexity, spherical Hadwiger theorem,
    spherical intrinsic volumes, conic valuations, conic intrinsic volumes}
}

\title{The Spherical Hadwiger Theorem}
\author{Suijie Wang \qquad Shengguo Wu}
\date{}

\begin{document}

\maketitle

\begingroup
\makeatletter
\renewcommand{\@makefntext}[1]{\parindent=0pt\noindent#1}
\makeatother
\renewcommand{\thefootnote}{}
\endgroup

\begin{abstract}
We prove the spherical Hadwiger theorem in every dimension: for
\(n\geq1\), every continuous \(SO(n+1)\)-invariant valuation on all
closed spherical convex sets in the unit sphere \(\Sphere^n\) is a unique
linear combination of the spherical intrinsic volumes \(V_0,\ldots,V_n\).
The same classification holds on sets contained in an open hemisphere,
with unique continuous invariant extension to the full domain.
We first classify simple invariant valuations by an independent induction
on dimension.  Averaging a signed cone identity gives the spherical
volume term; the remaining terms are signed coning transforms that
vanish by the lower-dimensional simple classification.
A second induction, together with extension to non-proper sets,
gives the general classification.
The cone--sphere correspondence yields, for \(d\geq2\), the corresponding
classification on all closed convex cones in \(\R^d\), without
normalization at the zero cone.
In both settings, \(SO\)-invariance implies \(O\)-invariance.
\end{abstract}

\section{Introduction}

For \(d\geq1\), Hadwiger's classical theorem states that every continuous valuation on
the nonempty compact convex subsets of \(\R^d\) that is invariant under
translations and the special orthogonal group \(SO(d)\) is a linear combination of the intrinsic
volumes \cite{Hadwiger,KlainRota,ChenHadwiger}.  It has motivated extensive work on
invariant valuations in other homogeneous geometries; see, for example,
\cite{SchneiderWeil,KlainHyperbolic,BernigFuSolanes}.  The spherical Hadwiger problem
asks for the analogous classification of continuous valuations on
spherical convex sets.  We establish this classification in every
dimension, assuming only invariance under the special orthogonal group.

Let \(\Sphere^n\) be the unit sphere in \(\R^{n+1}\).
We use the Euclidean inner product \(\langle\cdot,\cdot\rangle\)
and norm \(\|\cdot\|\).  For a nonempty set \(K\subset\Sphere^n\),
define its radial extension by
\[
 \widehat K:=\{r x:r\geq0,\ x\in K\}.
\]
Let \(\K(\Sphere^n)\) denote the empty set together with all closed
\(K\subset\Sphere^n\) for which \(\widehat K\) is convex.  We equip its
nonempty part with the Hausdorff topology induced by the spherical
geodesic distance and declare the empty set to be isolated.  A valuation
\(\mu:\K(\Sphere^n)\to\R\) satisfies \(\mu(\emptyset)=0\) and
\[
 \mu(K)+\mu(L)=\mu(K\cup L)+\mu(K\cap L)
\]
whenever \(K,L,K\cup L\in\K(\Sphere^n)\).  Let
\(v_0,\ldots,v_{n+1}\) be the conic intrinsic volumes in \(\R^{n+1}\),
and for nonempty \(K\in\K(\Sphere^n)\) set
\[
 V_j(K)=v_{j+1}(\widehat K)\qquad(0\leq j\leq n),
\]
with \(V_j(\emptyset)=0\).  Our main result is the following.

\begin{theorem}[Spherical Hadwiger theorem]\label{thm:main}
Let \(n\geq1\), and let
\(\mu:\K(\Sphere^n)\longrightarrow\R\) be a valuation that is
continuous in the spherical Hausdorff topology
and invariant under \(SO(n+1)\).  Then \(\mu\) has the unique representation
\[
 \mu=\sum_{j=0}^{n}\mu(S_j)V_j
 \qquad\text{on }\K(\Sphere^n).
\]
Here \(S_j\) is any great \(j\)-sphere, that is, the intersection of
\(\Sphere^n\) with a \((j+1)\)-dimensional linear subspace;
\(S_0\) is an antipodal pair and \(S_n=\Sphere^n\).
\end{theorem}

A closed spherical convex set is called
\emph{proper} if it is contained in an open hemisphere, and we write
\(\K_p(\Sphere^n)\) for the empty set together with all proper members of
\(\K(\Sphere^n)\).  The same classification holds on
\(\K_p(\Sphere^n)\), with unique continuous invariant extension to
\(\K(\Sphere^n)\); see Corollary~\ref{cor:spherical-form}.
The cone--sphere correspondence yields the classification on all
closed convex cones; see Corollary~\ref{cor:conic-form}.
A valuation is \emph{simple} if it vanishes on all lower-dimensional
sets in its domain.

The proof separates the simple and general classifications.
We first prove, independently of the extension results, that every
continuous, simple, \(SO(n+1)\)-invariant valuation
\(\eta:\K(\Sphere^n)\to\R\) satisfies
\(\eta=\eta(\Sphere^n)V_n\).
The argument is by induction on dimension.
Averaging a signed cone identity over an auxiliary vertex gives this
formula on proper simplices: the constant term yields
\(\eta(\Sphere^n)V_n(P)\), while the averaged face terms are signed
coning transforms on great hyperspheres.
Each transform is continuous, simple and invariant, and has value zero
on the whole hypersphere because opposite hemispheres are rotationally
congruent.  The lower-dimensional simple classification therefore makes
every transform vanish.
Triangulation, approximation and finite cutting give the formula on
the full domain.

Once the simple classification is established in every dimension,
we prove the general classification on proper spherical convex sets
by a second induction.
The lower-dimensional classification determines the restriction to
each great hypersphere.  Subtracting the corresponding intrinsic-volume
terms leaves a simple valuation; compare
\cite[Section~2]{KlainHadwiger} for the Euclidean reduction.
Its continuous simple extension to \(\K(\Sphere^n)\) is covered by the
classification already proved.
Finally, uniqueness of algebraic extension transfers the resulting
expansion from the proper domain to all closed spherical convex sets,
and evaluation on great spheres identifies its coefficients.

According to Schneider, the spherical Hadwiger problem was probably first
formulated, in an equivalent conic form, by McMullen in 1974
\cite[Section~3.3]{SchneiderCones}.  One formulation asks whether every
continuous \(SO(d)\)-invariant valuation on the class of closed spherical
convex sets in \(\Sphere^{d-1}\) is a linear combination of the spherical
intrinsic volumes.  Schneider also formulated the related question of
whether every increasing \(SO(d)\)-invariant valuation on proper spherical
convex bodies in \(\Sphere^{d-1}\) is a linear combination of the
spherical quermassintegrals with nonnegative coefficients
\cite[Section~3.3]{SchneiderCones}.  Both versions were recorded by Gruber and Schneider as Problem~49
\cite[Problem~49]{GruberSchneider}.  The continuous version was later recorded
by McMullen and Schneider \cite[Problem~15.5]{McMullenSchneider}.  The related
question whether every
continuous, simple valuation invariant under the orthogonal group
\(O(d)\) is a multiple of
spherical volume was recorded by Gruber and Schneider as Problem~74
\cite[Problem~74]{GruberSchneider}.  The cases \(\Sphere^1\) and
\(\Sphere^2\) were known
\cite[Proposition~11.2.2 and Theorem~11.3.1]{KlainRota}.
The higher-dimensional simple case was also listed as open by Ludwig
\cite[Section~4]{LudwigValuationTheory}.  The general continuous
classification on all closed spherical convex sets in higher dimensions
was also recorded as open in \cite[p.~263]{SchneiderWeil}.

Several partial results were known under additional hypotheses.
Schneider characterized spherical volume on spherical polytopes among
nonnegative \(SO(d)\)-invariant valuations that vanish on lower-dimensional
spherical sets \cite[Theorem~6.2]{SchneiderCurvature}.  Glasauer developed
the local integral geometry of spherically convex bodies
\cite{GlasauerThesis}.  Smooth invariant valuations on the sphere
were classified by Hack \cite[Theorem~4.1.3]{Hack}.  For the general
theory of smooth valuations on manifolds, see
\cite{AleskerManifoldsII}.
The classification for continuous valuations admitting nonnegative,
rotation-covariant local extensions follows from
Glasauer's characterization of spherical curvature measures
\cite[Satz~4.2.2]{GlasauerThesis}; see also
\cite[Theorem~6.5.4]{SchneiderWeil} and the discussion in
\cite[Section~5.6]{AmelunxenLotzMcCoyTropp}.  The combinatorics and
integral geometry of conic intrinsic volumes were developed further in
\cite{AmelunxenLotz,SchneiderCones}.  Related fan valuations and their
connection with spherical intrinsic volumes were studied in
\cite{BackmanManeckeSanyal}.

Independently and essentially simultaneously, Knoerr
\cite{KnoerrSpherical} proved the corresponding classification for
spherical polytopes contained in an open hemisphere.  Approximation
gives the proper convex-body classification, and the uniqueness of
algebraic extension in Lemma~\ref{lem:algebraic-extension} then yields
the classification on all closed spherical convex sets.
His proof draws on Alesker's quasi-smooth valuation theory
\cite{AleskerManifoldsI} and uses \(\operatorname{GL}(n+1,\R)\)-smoothing,
gnomonic projection, and his Euclidean polytopal Hadwiger theorem
\cite{KnoerrEuclidean}.  The present paper gives an independent geometric proof based on a
signed cone identity, averaging, and a dimension-reducing coning
transform.

Subsequent work by Fu \cite{FuMonotone} and Lotz \cite{LotzMonotone}
gives monotone refinements of the spherical and conic classifications,
respectively, with automatic continuity and orthogonal invariance.
Fu also obtains hyperbolic and elliptic analogues; for a hyperbolic
polytopal classification, see also Knoerr~\cite{KnoerrHyperbolic}.
Lotz gives a Grassmann-angle characterization on nonzero pointed cones that
settles a conjecture of McMullen.  Lotz subsequently gave an elementary
proof of the continuous classification using signed orthoscheme
decompositions \cite{LotzOrthoschemes}.

Section~\ref{sec:preliminaries} collects the geometric and
valuation-theoretic preliminaries.
Section~\ref{sec:simple-classification} proves the classification of
simple invariant valuations by signed coning and averaging.
Section~\ref{sec:general-classification} presents the extension principles,
derives the general spherical classification, and obtains the conic
formulation.
Section~\ref{sec:extension} proves the extension results used in the
general classification.

\section{Preliminaries}\label{sec:preliminaries}

\subsection{Spherical convex sets and cones}

For every nonzero linear subspace \(W\subset\R^{n+1}\), write
\(\Sphere(W):=W\cap\Sphere^n\); this is a great hypersphere when
\(\dim W=n\).  Let \(\K(\Sphere(W))\) be the family
of spherical links \(C\cap\Sphere(W)\), where \(C\subseteq W\) is a
closed convex cone, including the empty set corresponding to
\(C=\{0\}\).  Let \(\K_p(\Sphere(W))\) consist of the empty set and
the members contained in an open hemisphere of \(\Sphere(W)\).  When
\(E=\Sphere(W)\), these families are abbreviated to \(\K(E)\) and
\(\K_p(E)\).
Neither full-dimensionality nor containment in an open hemisphere is
required in \(\K(\Sphere(W))\).

For \(d\geq1\), let \(\C_d\) denote the set of all nonempty closed convex cones
in \(\R^d\).  The cone--sphere correspondence is the bijection
\[
 \C_d\longrightarrow\K(\Sphere^{d-1}),\qquad
 C\longmapsto
 \begin{cases}
  C\cap\Sphere^{d-1},&C\neq\{0\},\\
  \emptyset,&C=\{0\}
 \end{cases}
\]
whose inverse is \(K\mapsto\widehat K\), with
\(\widehat\emptyset=\{0\}\).

A convex cone \(C\) is \emph{pointed} if
\(C\cap(-C)=\{0\}\).
Write \(\conv\) for the Euclidean convex hull.
For nonempty \(K\in\K(\Sphere^n)\), the cone \(\widehat K\) is pointed if and only if
\(0\notin\conv K\), which, by strict separation
\cite[Chapter~1]{SchneiderConvexBodies}, is equivalent to \(K\) being contained
in an open hemisphere.

Define
\[
 \dim C:=\dim\operatorname{span}C\quad(C\in\C_d),\qquad
 \dim K:=\dim\widehat K-1
\]
for nonempty \(K\).  This convention is independent of the ambient
linear subspace.

Write \(\pos A\) for the set of finite nonnegative linear combinations
of elements of \(A\), including zero.
In particular, \(\pos\emptyset=\{0\}\).
A cone is \emph{polyhedral} if it is an
intersection of finitely many closed linear halfspaces, equivalently
the positive hull of finitely many vectors
\cite[Chapter~1]{SchneiderConvexBodies}.  A \emph{spherical polytope}
is the link of a polyhedral cone; it need not be proper.

Faces of a spherical polytope are the spherical links of faces of its cone.
If a spherical polytope \(P\) has \(\dim P=n\), a \emph{facet} of
\(P\) means \(G\cap\Sphere^n\), where \(G\) is a
codimension-one face of \(\widehat P\).  Its supporting great
hypersphere is \(\operatorname{span}(G)\cap\Sphere^n\).

For a finite set \(A\subset\Sphere^n\), define its spherical convex
hull by
\[
 \convS A:=\pos A\cap\Sphere^n.
\]
The positive hull is polyhedral and hence closed.  This definition
uses the cone--sphere convention above and does not require an open
hemisphere; in particular, \(\convS\emptyset=\emptyset\).
For linearly independent unit vectors \(p_0,\ldots,p_m\), the set
\(\convS\{p_0,\ldots,p_m\}\) is called a \emph{nondegenerate spherical
\(m\)-simplex}.  Such a simplex is proper.

\subsection{Hausdorff topologies and compactness}

We use the term \emph{closed} for spherical sets.  Let \(\operatorname{dist}_s(x,y)=\arccos\langle x,y\rangle\) be
spherical distance, and let \(d_s\) be the induced Hausdorff metric.
We equip the nonempty part of \(\K(\Sphere(W))\) with this metric
and adjoin \(\emptyset\) as an isolated point.  Interiors are taken
relative to the ambient sphere.  Let \(\overline B_d\) be the closed
Euclidean unit ball and \(d_{\mathrm H}\) Euclidean Hausdorff distance.
For cones, use
\[
 d_c(C,D)=d_{\mathrm H}(C\cap\overline B_d,D\cap\overline B_d).
\]

\begin{lemma}\label{lem:link-homeomorphism}
For every \(d\geq1\), the cone--sphere correspondence is a homeomorphism
from the conic Hausdorff topology to the spherical Hausdorff topology,
with the zero cone and the empty set isolated.
\end{lemma}

\begin{proof}
For nonzero cones, put
\(h=d_c(C,D)\) and
\(\theta=d_s(C\cap\Sphere^{d-1},D\cap\Sphere^{d-1})\).
Matching points at equal radii gives \(h\leq2\sin(\theta/2)\).
Conversely, if \(h<1\), each unit vector \(a\in C\) lies within
\(h\) of some \(tb\in D\cap\overline B_d\), where \(t>0\) and
\(b\in D\cap\Sphere^{d-1}\).  The angle \(\alpha\) between
\(a,b\) is acute and satisfies
\(\sin\alpha\leq\|a-tb\|\leq h\).
Interchanging \(C,D\) gives \(\theta\leq\arcsin h\).
These bounds prove continuity in both directions.  Every nonzero
cone has distance \(1\) from the zero cone, so the latter is isolated.
\end{proof}

\begin{lemma}\label{lem:spherical-hyperspace-compactness}
For every nonzero linear subspace \(W\subset\R^{n+1}\), the space
\(\K(\Sphere(W))\), with the empty set isolated, is compact.
\end{lemma}

\begin{proof}
The hyperspace of nonempty closed subsets of \(\Sphere(W)\) is
compact.  Its spherical-convex subfamily is closed: if \(K_r\to K\),
normalized nonzero positive combinations of points of \(K\) are limits
of the corresponding combinations in \(K_r\), hence belong to \(K\).
Thus \(\widehat K\) is convex; it is closed because \(K\) is compact.
Adjoining the isolated empty set preserves compactness.
\end{proof}

\subsection{Valuations and intrinsic volumes}

If \(E=\Sphere(W)\), a map \(\mu:\K(E)\to\R\) is a valuation if
\(\mu(\emptyset)=0\) and
\[
 \mu(K)+\mu(L)=\mu(K\cup L)+\mu(K\cap L),
\]
whenever \(K,L,K\cup L\in\K(E)\).  The same definition applies with
\(\K(E)\) replaced by \(\K_p(E)\).  If \(\dim W=m+1\), a valuation on
\(\K(E)\) or \(\K_p(E)\) is called \emph{simple} if
\[
 \mu(K)=0
 \qquad\text{whenever }K\neq\emptyset
 \text{ and }\dim K\leq m-1.
\]

For a Euclidean subspace \(W\), write \(O(W)\) and \(SO(W)\) for its
orthogonal and special orthogonal groups.  A valuation on
\(\K(\Sphere(W))\) or \(\K_p(\Sphere(W))\) is \(G\)-invariant, for
\(G\leq O(W)\), if \(\mu(gK)=\mu(K)\) for all \(g\in G\) and all
sets \(K\) in its domain.

A map \(\Phi:\C_d\to\R\) is a
\emph{conic valuation} if
\[
\Phi(C)+\Phi(D)=\Phi(C\cup D)+\Phi(C\cap D),
\]
whenever \(C,D\in\C_d\) and \(C\cup D\) is convex.
Such a valuation is called \emph{normalized} if
\(\Phi(\{0\})=0\).

Let \(v_0,\ldots,v_d\) be the conic intrinsic volumes in
\(\R^d\); see \cite[Section~2]{AmelunxenLotz},
\cite[Section~5.1]{AmelunxenLotzMcCoyTropp}, and
\cite[Section~6.5]{SchneiderWeil}.
Write \(\delta_{jk}\) for the Kronecker delta.
They are continuous \(O(d)\)-invariant conic valuations, intrinsic
under isometric linear embeddings, and satisfy
\[
 v_j(C)=0\quad\text{if }\dim C<j,
 \qquad v_j(L_k)=\delta_{jk},
\]
where \(L_k\) is a \(k\)-dimensional subspace.
On a ray,
\(v_0=v_1=1/2\), with all other intrinsic volumes zero;
see \cite[Sections~2.3 and~4.2]{SchneiderCones}.
Thus \(v_j(\{0\})=0\) for \(j\geq1\), whereas \(v_0(\{0\})=1\).

The cone--sphere correspondence defines continuous
\(O(n+1)\)-invariant valuations
\[
 V_j(K)=v_{j+1}(\widehat K),\qquad 0\leq j\leq n.
\]
Let \(\operatorname{vol}_n\) denote spherical \(n\)-dimensional volume.
Our normalization gives
\[
 V_0(\{p\})=\tfrac12,
 \qquad V_j(S_k)=\delta_{jk},
 \qquad V_n(K)=\frac{\operatorname{vol}_n(K)}
 {\operatorname{vol}_n(\Sphere^n)}.
\]
Here \(S_k\) is any great \(k\)-sphere.  In particular, \(V_n\)
is simple and \(V_0\) is not the usual Euler characteristic.

\section{Simple invariant valuations}\label{sec:simple-classification}

We classify simple invariant valuations on \(\K(\Sphere^n)\).

\begin{theorem}\label{prop:simple-classification}
Let \(n\geq1\).  Every continuous, simple,
\(SO(n+1)\)-invariant valuation
\(\eta:\K(\Sphere^n)\to\R\) satisfies
\[
 \eta=\eta(\Sphere^n)V_n.
\]
\end{theorem}

We prove the theorem by induction on dimension, beginning with the circle.

\subsection{The circle and the induction hypothesis}

\begin{lemma}\label{lem:simple-circle}
Every continuous, simple, \(SO(2)\)-invariant valuation
\(\eta:\K(\Sphere^1)\to\R\) satisfies
\(\eta=\eta(\Sphere^1)V_1\).
\end{lemma}

\begin{proof}
Every nonempty proper spherical convex set in \(\Sphere^1\) is a point
or a closed arc of length less than \(\pi\).
By rotation invariance, the value of \(\eta\) on such an arc depends
only on its length; write this value as \(f(t)\), \(0\leq t<\pi\).
Continuity of \(\eta\) gives continuity of \(f\), and simplicity gives
\(f(0)=0\).  Splitting an arc into two adjacent arcs yields
\[
 f(s+t)=f(s)+f(t)
 \qquad(s,t\geq0,\ s+t<\pi).
\]
Fix \(\ell\in(0,\pi)\).  Subdivision gives
\(f(q\ell)=qf(\ell)\) for every nonnegative rational \(q\) with
\(q\ell<\pi\).  Continuity therefore gives \(f(t)=at\) on
\([0,\pi)\), where \(a=f(\ell)/\ell\).

Set \(\rho=\eta-2\pi aV_1\).  This valuation is simple and vanishes
on every proper set.  For \(\varepsilon\in\{-1,1\}^2\), let
\[
 O_\varepsilon
 :=\{y\in\Sphere^1:\varepsilon_jy_j\geq0,\ j=1,2\}.
\]
Cutting any \(K\in\K(\Sphere^1)\) successively by the two coordinate
great \(0\)-spheres gives
\[
 \rho(K)=\sum_{\varepsilon\in\{-1,1\}^2}
          \rho(K\cap O_\varepsilon).
\]
The intersections arising at each cut are empty or zero-dimensional
and have value zero.  Each \(O_\varepsilon\) is a closed quarter
circle, so every term in the sum is zero.
Thus \(\eta=2\pi aV_1\) on the full domain.  Evaluation on
\(\Sphere^1\) identifies the coefficient as \(\eta(\Sphere^1)\).
\end{proof}

For the rest of the inductive step, fix \(n\geq2\) and assume
Theorem~\ref{prop:simple-classification} in dimension \(n-1\).
Let \(\eta:\K(\Sphere^n)\to\R\) be a continuous, simple,
\(SO(n+1)\)-invariant valuation.
We use \(\sgn t\) for the sign of a real number, with \(\sgn0=0\).
The coning transform below will eliminate the face terms in the
averaged signed cone identity.

\subsection{The signed coning transform}

Let \(W\subset\R^{n+1}\) be an \(n\)-dimensional subspace, set
\(E=\Sphere(W)\), and choose \(\nu_E\in W^\perp\cap\Sphere^n\).
For Euclidean sets \(A,B\), write
\(A+B:=\{a+b:a\in A,\ b\in B\}\) for their Minkowski sum.
For \(x\in\Sphere^n\setminus E\) and \(Q\in\K(E)\), define
\[
 x*Q:=(\widehat Q+\R_{\geq0}x)\cap\Sphere^n.
\]
The direct-sum decomposition \(\R^{n+1}=W\oplus\R x\) shows
that the cone in this definition is closed.  In particular,
\(x*Q\in\K(\Sphere^n)\), and \(x*\emptyset=\{x\}\).

\begin{lemma}\label{lem:coning-continuity}
If \(x_r\to x\) in \(\Sphere^n\setminus E\) and
\(Q_r\to Q\) in \(\K(E)\), then \(x_r*Q_r\to x*Q\).
\end{lemma}

\begin{proof}
For nonempty \(Q\), the set \(x*Q\) is the image of
\(\{x\}\times Q\times[0,1]\) under
\[
 \varphi(x,q,s)=\frac{(1-s)q+sx}{\|(1-s)q+sx\|}.
\]
Since \(x\notin E\), the denominator never vanishes.
The map is therefore uniformly continuous on
\(U\times E\times[0,1]\) for every closed \(U\subset\Sphere^n\)
disjoint from \(E\).  Taking such a \(U\) containing
\(x\) and all \(x_r\) for sufficiently large \(r\) proves convergence of the
images in Hausdorff distance.  If \(Q=\emptyset\), then
\(Q_r=\emptyset\) eventually and \(x_r*Q_r=\{x_r\}\to\{x\}\).
\end{proof}

Let \(\sigma_n\) denote spherical volume normalized by
\(\sigma_n(\Sphere^n)=1\); thus \(V_n(K)=\sigma_n(K)\).
Define the signed coning transform by
\begin{equation}
 \bigl(\mathcal T_{E,\nu_E}\eta\bigr)(Q)
 :=\int_{\Sphere^n\setminus E}
 \sgn\langle x,\nu_E\rangle\,
 \eta(x*Q)\,\dd\sigma_n(x).
 \label{eq:cone-transform}
\end{equation}
For fixed \(Q\), the integrand is continuous off \(E\).
Extending it by zero on \(E\) gives a Borel function, bounded
uniformly in \(Q\) because \(\eta\) is continuous on the compact
space \(\K(\Sphere^n)\)
(Lemma~\ref{lem:spherical-hyperspace-compactness}).  Thus the integral
is well defined.  It is the signed hemispherical transform of
\(x\mapsto\eta(x*Q)\), evaluated at \(\nu_E\);
compare \cite[Section~3.2, equation~(3.3)]{RubinRadonNotes}.

\begin{proposition}\label{prop:cone-transform}
The map
\[
 \bigl(\mathcal T_{E,\nu_E}\eta\bigr):\K(E)\longrightarrow\R
\]
is a continuous, simple, \(SO(W)\)-invariant valuation.
\end{proposition}

\begin{proof}
For fixed \(x\notin E\), every vector has a unique decomposition
\(w+tx\), with \(w\in W\) and \(t\in\R\).  A vector in the cone
over \(x*Q\) has \(w\in\widehat Q\) and \(t\geq0\).  Hence, for
\(Q,R\in\K(E)\) with \(Q\cup R\in\K(E)\),
\[
 (x*Q)\cap(x*R)=x*(Q\cap R),
 \qquad
 (x*Q)\cup(x*R)=x*(Q\cup R).
\]
In particular, if \(Q\cap R=\emptyset\), then
\(\widehat Q\cap\widehat R=\{0\}\), so
\((x*Q)\cap(x*R)=x*\emptyset=\{x\}\).
Simplicity gives \(\eta(\{x\})=0\).  Thus integration of the
valuation identity gives additivity, including this empty-intersection
case, and the transform has value zero at \(\emptyset\).
For nonempty \(Q\), the bound \(\dim(x*Q)\leq\dim Q+1\)
gives simplicity.
Continuity follows from Lemma~\ref{lem:coning-continuity}, boundedness
of \(\eta\), and dominated convergence.
Finally, extend \(h\in SO(W)\) to \(g\in SO(n+1)\) fixing \(\nu_E\).
Then \(g(x*Q)=(gx)*(hQ)\) and
\(\langle gx,\nu_E\rangle=\langle x,\nu_E\rangle\); the change of
variables \(x=gy\) in \eqref{eq:cone-transform} proves invariance.
\end{proof}

\begin{lemma}[Vanishing of the coning transform]\label{lem:transform-zero}
The valuation \(\mathcal T_{E,\nu_E}\eta\) is identically zero.
\end{lemma}

\begin{proof}
For \(x\notin E\), the set \(x*E\) is the closed hemisphere bounded
by \(E\) on the side containing \(x\).
Rotation by \(\pi\) in a two-plane spanned by \(\nu_E\) and a unit
vector of \(W\) exchanges the two closed hemispheres and belongs to
\(SO(n+1)\), so their \(\eta\)-values agree.  Since each open hemisphere has
\(\sigma_n\)-measure \(1/2\), their contributions to the signed
integral cancel:
\[
 \bigl(\mathcal T_{E,\nu_E}\eta\bigr)(E)=0.
\]
By Proposition~\ref{prop:cone-transform}, the transform is a
continuous, simple, \(SO(W)\)-invariant valuation on the full domain
\(\K(E)\).  The induction hypothesis in dimension \(n-1\) gives
\[
 \mathcal T_{E,\nu_E}\eta
 =\bigl(\mathcal T_{E,\nu_E}\eta\bigr)(E)V_{n-1}^E=0,
\]
where \(V_{n-1}^E\) denotes normalized spherical volume on \(E\).
\end{proof}

\subsection{The signed cone identity and its average}\label{subsec:signed-identity}

We first record how simple valuations act on identities between
indicator functions.

The relation between valuation extensions and linear relations among
indicator functions is classical; see
\cite[Theorem~1 and Section~5]{GroemerExtension}.  For extension to
finite unions in the spherical setting, see also
\cite[Theorem~1.3]{KlainHyperbolic}.

Write \(\indicator_A\) for the indicator function of a set \(A\).

\begin{lemma}\label{lem:finite-cuts}
Let \(\rho\) be a simple valuation on \(\K(\Sphere^n)\), and let
\(P_1,\ldots,P_m\) be spherical polytopes and
\(\alpha_1,\ldots,\alpha_m\in\R\).  If the pointwise identity
\[
 \sum_{i=1}^m\alpha_i\indicator_{P_i}=0
\]
holds at every point of
\(\Sphere^n\setminus\bigcup_{a=1}^r H_a\), for some finite list of
great hyperspheres \(H_1,\ldots,H_r\), then
\[
 \sum_{i=1}^m\alpha_i\rho(P_i)=0.
\]
\end{lemma}

\begin{proof}
Choose a finite arrangement of great hyperspheres containing the
exceptional hyperspheres, all facet-supporting hyperspheres of the
full-dimensional \(P_i\), a hypersphere containing each nonempty
lower-dimensional \(P_i\), and the coordinate hyperspheres.
Let \(\mathcal U\) be its open chambers, the connected components
of the complement of the arrangement.  On each \(U\), every
\(\indicator_{P_i}\) has a constant value
\(\varepsilon_{i,U}\in\{0,1\}\).
Successive cutting by the arrangement and simplicity give
\begin{equation}
 \rho(P_i)=\sum_{U\in\mathcal U}
 \varepsilon_{i,U}\rho(\overline U).
 \label{eq:polytope-chamber-sum}
\end{equation}
Indeed, all lower-dimensional cut pieces and all intersections between
distinct chamber closures have value zero.  If \(\varepsilon_{i,U}=1\),
then \(P_i\cap\overline U=\overline U\); otherwise the intersection
is empty or lower-dimensional and also has value zero.
The assumed indicator identity gives
\(\sum_i\alpha_i\varepsilon_{i,U}=0\) on every chamber.
Substituting in \eqref{eq:polytope-chamber-sum} proves the claim.
\end{proof}

The following identity concerns the \(n+2\) simplices obtained by
omitting one vector from an \((n+2)\)-tuple.

Call \(q_0,\ldots,q_{n+1}\in\Sphere^n\) \emph{in general position}
if every \(n+1\) of them are linearly independent.
A hat over an entry denotes omission; determinants use column vectors
and the standard orientation of \(\R^{n+1}\).  For such a tuple set
\[
 \begin{aligned}
 P_i&=\convS\{q_0,\ldots,\widehat q_i,\ldots,q_{n+1}\},\\
 d_i&=(-1)^i\det(q_0,\ldots,\widehat q_i,\ldots,q_{n+1}).
 \end{aligned}
\]
All \(d_i\) are nonzero.  The exceptional set
\[
 Z:=\bigcup_{0\leq i<j\leq n+1}
 \left(\operatorname{span}\{q_k:k\neq i,j\}\cap\Sphere^n\right)
\]
is a finite union of great hyperspheres.

\begin{lemma}[Signed cone identity]\label{lem:radial-boundary-chain}
For every \(y\in\Sphere^n\setminus Z\),
\begin{equation}
 \sum_{i=0}^{n+1}\sgn(d_i)\indicator_{P_i}(y)
 =\gamma,
 \label{eq:radial-boundary-chain}
\end{equation}
where
\[
 \gamma=\gamma(q_0,\ldots,q_{n+1})=
 \begin{cases}
  1,&d_i>0\text{ for every }i,\\
  -1,&d_i<0\text{ for every }i,\\
  0,&\text{otherwise}.
 \end{cases}
\]
\end{lemma}

\begin{proof}
Put \(A=(q_0\ \cdots\ q_{n+1})\) and
\(\mathbf d=(d_0,\ldots,d_{n+1})^{\mathsf T}\).
Cofactor expansion gives \(A\mathbf d=0\).  Since
\(\operatorname{rank}A=n+1\), its kernel is \(\R\mathbf d\).
Fix \(y\notin Z\), and choose a solution \(t^0\) of \(At=y\).
All solutions are \(t=t^0+s\mathbf d\), \(s\in\R\), and the
nonnegative solutions are parametrized by
\[
 I_y:=\{s\in\R:t_i^0+s d_i\geq0\text{ for every }i\}.
\]
The \(i\)-th inequality reads \(s\geq-t_i^0/d_i\) when \(d_i>0\)
and \(s\leq-t_i^0/d_i\) when \(d_i<0\).  Thus \(I_y\) is a closed
interval, possibly empty or unbounded.

No feasible solution can have two zero coordinates: if \(t_i=t_j=0\),
then \(y=\sum_{k\neq i,j}t_kq_k\) would belong to \(Z\).
Every finite endpoint of \(I_y\) has a zero coordinate, since strict
positivity of all coordinates would make it an interior point.
Consequently, exactly one coordinate vanishes at each finite endpoint.
The interval cannot be a singleton, since then a lower and an upper
bound would both be attained, giving two zero coordinates.

By the definition of \(P_i\), membership \(y\in P_i\) is
equivalent to the existence of a nonnegative solution of \(At=y\)
with \(t_i=0\).  Since \(d_i\neq0\), this occurs at the unique
parameter \(s=-t_i^0/d_i\), which is a left endpoint when \(d_i>0\)
and a right endpoint when \(d_i<0\).  Thus each left endpoint
contributes \(+1\) to the sum in \eqref{eq:radial-boundary-chain},
and each right endpoint contributes \(-1\).

If the \(d_i\) have both signs, then \(I_y\) is empty or a bounded
nondegenerate interval; in the latter case, the two endpoint contributions
cancel.  If all \(d_i\) have the same sign, then \(I_y\) is a nonempty
half-line whose unique endpoint contributes that common sign.
In either case, the sum equals \(\gamma\).
\end{proof}

\begin{corollary}[Valuation identity]\label{cor:signed-valuation-identity}
Let \(\rho\) be a simple valuation on \(\K(\Sphere^n)\).
For every tuple in general position,
\begin{equation}
 \sum_{i=0}^{n+1}\sgn(d_i)\rho(P_i)
 =\gamma\rho(\Sphere^n).
 \label{eq:signed-valuation-identity}
\end{equation}
\end{corollary}

\begin{proof}
Apply Lemma~\ref{lem:finite-cuts} to
\eqref{eq:radial-boundary-chain}, with the constant term written as
\(\gamma\indicator_{\Sphere^n}\).
\end{proof}

If the complete tuple lies in a common open hemisphere, taking the
inner product with its center in \(\sum_i d_iq_i=0\) shows that the
\(d_i\) have both signs.  Thus \(\gamma=0\), and
\eqref{eq:signed-valuation-identity} is a relation involving only
proper sets, with no condition on \(\rho(\Sphere^n)\).
Each individual \(P_i\) is proper even when \(\gamma\neq0\).
We retain the constant term: its average below is
exactly the spherical volume contribution.

\begin{lemma}\label{lem:simplex-volume}
Every proper nondegenerate spherical \(n\)-simplex \(P\) satisfies
\[
 \eta(P)=\eta(\Sphere^n)V_n(P).
\]
\end{lemma}

\begin{proof}
Write \(P=\convS\{p_0,\ldots,p_n\}\) with independent unit generators
and put \(D=\det(p_0,\ldots,p_n)\).
For \(0\leq i\leq n\), set
\[
 W_i=\operatorname{span}(p_0,\ldots,\widehat p_i,\ldots,p_n),
 \quad E_i=\Sphere(W_i),
 \quad Q_i=\convS\{p_0,\ldots,\widehat p_i,\ldots,p_n\},
\]
and choose \(\nu_i\in W_i^\perp\cap\Sphere^n\) with
\(\det(\nu_i,p_0,\ldots,\widehat p_i,\ldots,p_n)>0\).
For \(x\notin E_i\), define
\[
 F_i(x)=
 \sgn\det(x,p_0,\ldots,\widehat p_i,\ldots,p_n)\,
 \eta\bigl(\convS\{x,p_0,\ldots,\widehat p_i,\ldots,p_n\}\bigr).
\]
For \(x\notin E_i\), the simplex in this formula is \(x*Q_i\)
and its determinant sign is \(\sgn\langle x,\nu_i\rangle\).
On \(E_i\), set \(F_i=0\).  The measurability and boundedness
established for \eqref{eq:cone-transform} apply, and
Lemma~\ref{lem:transform-zero} gives
\[
 \int_{\Sphere^n}F_i(x)\,\dd\sigma_n(x)
 =\bigl(\mathcal T_{E_i,\nu_i}\eta\bigr)(Q_i)=0.
\]

If \(x\notin\bigcup_iE_i\), the tuple
\((x,p_0,\ldots,p_n)\) is in general position.
Corollary~\ref{cor:signed-valuation-identity} yields
\begin{equation}
 \sgn D\,\eta(P)
 +\sum_{i=0}^n(-1)^{i+1}F_i(x)
 =\eta(\Sphere^n)\,\gamma(x,p_0,\ldots,p_n).
 \label{eq:apex-identity}
\end{equation}
Write \(x=-\sum_{i=0}^n t_ip_i\), which uniquely defines the real
coefficients \(t_i\).  The matrix \((x\ p_0\ \cdots\ p_n)\) has
one-dimensional kernel spanned by
\((1,t_0,\ldots,t_n)^{\mathsf T}\).  Its cofactor vector
\((d_0,\ldots,d_{n+1})^{\mathsf T}\) lies in the same kernel and
has first coordinate \(d_0=D\).  It therefore equals
\(D(1,t_0,\ldots,t_n)^{\mathsf T}\), giving
\(d_{i+1}=Dt_i\).  The cofactors all have the same sign exactly when
every \(t_i>0\), equivalently \(x\in-\operatorname{int}P\).
Consequently,
\begin{equation}
 \gamma(x,p_0,\ldots,p_n)
 =\sgn D\,\indicator_{-\operatorname{int}P}(x)
 \quad\text{for almost every }x.
 \label{eq:gamma-volume}
\end{equation}
The exceptional set \(\bigcup_iE_i\) has spherical measure zero
and contains the boundary of \(P\).  Antipodal maps preserve spherical
volume, so \(\sigma_n(-\operatorname{int}P)=\sigma_n(P)\).
Integrate \eqref{eq:apex-identity}.  The face terms vanish, while
\eqref{eq:gamma-volume} gives
\[
 \sgn D\,\eta(P)
 =\eta(\Sphere^n)\,\sgn D\,\sigma_n(P)
 =\eta(\Sphere^n)\,\sgn D\,V_n(P).
\]
Since \(D\neq0\), this proves the claim.
\end{proof}

\subsection{Completion of the simple classification}

A finite \emph{face-to-face triangulation} of an \(n\)-dimensional
spherical polytope \(P\) is a finite collection of nondegenerate spherical
\(n\)-simplices whose union is \(P\), such that any two distinct
simplices intersect, if at all, in a common proper face.

\begin{lemma}\label{lem:spherical-triangulation}
Every full-dimensional proper spherical polytope \(P\in\K(\Sphere^n)\)
admits a finite face-to-face triangulation into proper nondegenerate
spherical \(n\)-simplices.
\end{lemma}

\begin{proof}
Choose \(u\in\Sphere^n\) with \(\langle u,p\rangle>0\) for every
\(p\in P\).  The gnomonic projection
\(p\mapsto p/\langle u,p\rangle\) sends \(P\) to a compact
Euclidean \(n\)-polytope in \(\{z:\langle u,z\rangle=1\}\);
see \cite[Section~4]{BesauSchuster}.
A face-to-face triangulation of this polytope
\cite[Proposition~2.2.4]{DeLoeraRambauSantos} pulls back under radial
projection to the required triangulation.  Faces and intersections are
preserved, and affine independence in that hyperplane implies linear
independence of the vertex vectors.  All simplices remain in \(P\)
and hence are proper.
\end{proof}

\begin{lemma}\label{lem:polytope-density}
Spherical polytopes are dense in \(\K(\Sphere^n)\) for the spherical
Hausdorff topology.  Equivalently, polyhedral cones are dense among
all closed convex cones in the conic Hausdorff topology.
\end{lemma}

\begin{proof}
For nonempty \(K\), choose a finite \(1/r\)-net \(F_r\subset K\).
Then \(P_r=\convS F_r\) is a spherical polytope with
\(F_r\subset P_r\subset K\), so \(d_s(P_r,K)\leq1/r\).
The empty set is already a spherical polytope.
\end{proof}

\begin{proof}[Proof of Theorem~\ref{prop:simple-classification}]
The case \(n=1\) is Lemma~\ref{lem:simple-circle}.
For the inductive step, let \(n\geq2\) and retain the notation and
induction hypothesis above.

Set \(\rho=\eta-\eta(\Sphere^n)V_n\).
This is a continuous simple valuation, and
Lemma~\ref{lem:simplex-volume} shows that it vanishes on every proper
nondegenerate \(n\)-simplex.
Let \(P\) be a full-dimensional proper spherical polytope, and choose
a triangulation \(P_1,\ldots,P_N\) as in
Lemma~\ref{lem:spherical-triangulation}.
Away from the finitely many supporting great hyperspheres of the
simplex facets, one has
\[
 \indicator_P=\sum_{j=1}^N\indicator_{P_j}.
\]
Lemma~\ref{lem:finite-cuts} therefore gives
\[
 \rho(P)=\sum_{j=1}^N\rho(P_j)=0.
\]
Lower-dimensional polytopes have value zero by simplicity.
For every proper \(K\), the approximating polytopes in
Lemma~\ref{lem:polytope-density} lie inside \(K\), and hence are proper.
Continuity gives \(\rho(K)=0\) on the proper domain.

To pass to the full domain, for
\(\varepsilon\in\{-1,1\}^{n+1}\) define the closed spherical orthant
\[
 O_\varepsilon
 :=\{y\in\Sphere^n:\varepsilon_jy_j\geq0,\ 1\leq j\leq n+1\}.
\]
Each orthant lies in the open hemisphere centered at
\(\varepsilon/\sqrt{n+1}\) and is therefore proper: for
\(y\in O_\varepsilon\),
\[
 \left\langle y,\frac{\varepsilon}{\sqrt{n+1}}\right\rangle
 =\frac{\sum_{j=1}^{n+1}|y_j|}{\sqrt{n+1}}
 \geq\frac1{\sqrt{n+1}}>0.
\]
For arbitrary \(K\in\K(\Sphere^n)\), cut \(K\) successively by the
coordinate great hyperspheres.
At each cut, the valuation identity expresses the value of the
current set as the sum of the values of its two closed pieces,
because their intersection is empty or lower-dimensional.
Iterating gives
\[
 \rho(K)=\sum_{\varepsilon\in\{-1,1\}^{n+1}}
                \rho(K\cap O_\varepsilon).
\]
Every term is zero, since \(K\cap O_\varepsilon\) is proper or empty.
Thus \(\rho=0\) on \(\K(\Sphere^n)\), completing the induction.
\end{proof}

\section{The spherical and conic classifications}\label{sec:general-classification}

We deduce the general classification from the simple case.

\subsection{Extension principles}\label{subsec:extension-principles}

We use the following two results, proved independently in
Section~\ref{sec:extension}.  Algebraic uniqueness will transfer
identities to the full domain; continuity of the simple extension
will allow us to apply Theorem~\ref{prop:simple-classification}.

\begin{lemma}[Algebraic extension]\label{lem:algebraic-extension}
Every valuation \(\mu:\K_p(\Sphere^n)\to\R\) has a unique valuation
extension \(\bar\mu:\K(\Sphere^n)\to\R\).  Moreover,
\(SO(n+1)\)-invariance and simplicity pass from \(\mu\) to \(\bar\mu\).
\end{lemma}

\begin{proposition}\label{prop:continuous-simple-extension}
Every continuous simple valuation on \(\K_p(\Sphere^n)\) admits a
unique continuous simple extension to \(\K(\Sphere^n)\).  If the
original valuation is \(SO(n+1)\)-invariant, so is its extension.
\end{proposition}

\subsection{The spherical classification}\label{subsec:extension-induction}

We first classify valuations on the proper domain by subtracting the
lower-dimensional intrinsic-volume terms.

\begin{proposition}[Classification on the proper domain]\label{prop:proper-classification}
Let \(n\geq1\).  Every continuous \(SO(n+1)\)-invariant valuation
\(\mu:\K_p(\Sphere^n)\to\R\) has a unique representation
\[
 \mu=\sum_{j=0}^{n}a_jV_j.
\]
\end{proposition}

\begin{proof}
We first prove uniqueness in every dimension.
Suppose that \(\sum_{j=0}^{n}b_jV_j=0\) on the proper domain.
This linear combination and the zero valuation are both extensions
of the zero valuation to \(\K(\Sphere^n)\).
Lemma~\ref{lem:algebraic-extension} gives
\(\sum_{j=0}^{n}b_jV_j=0\) on the full domain.
Evaluating on a great \(k\)-sphere and using
\(V_j(S_k)=\delta_{jk}\) yields \(b_k=0\) for every \(k\).

We prove existence by induction on \(n\).
For \(n=1\), fix \(p\in\Sphere^1\).  Rotation invariance makes
\(\mu(\{p\})\) the common value on all points.
Since \(V_0(\{p\})=1/2\), the valuation
\(\eta:=\mu-2\mu(\{p\})V_0\) vanishes on points and hence is simple.
Proposition~\ref{prop:continuous-simple-extension} gives a continuous
simple invariant extension \(\bar\eta\) to \(\K(\Sphere^1)\).
Theorem~\ref{prop:simple-classification} gives
\(\bar\eta=\bar\eta(\Sphere^1)V_1\).
Restriction to the proper domain proves the required expansion.

Now let \(n\geq2\), and assume the proposition in dimension \(n-1\).
Fix a great hypersphere \(E=\Sphere(W)\), where \(\dim W=n\).
Every \(h\in SO(W)\) extends by the identity on \(W^\perp\) to
an element of \(SO(n+1)\), so the restriction of \(\mu\) to
\(\K_p(E)\) is continuous and \(SO(W)\)-invariant.
The induction hypothesis and intrinsicity of the \(V_j\) give
unique coefficients \(a_0,\ldots,a_{n-1}\) such that
\[
 \mu=\sum_{j=0}^{n-1}a_jV_j\qquad\text{on }\K_p(E).
\]
These coefficients work on every great hypersphere.
Indeed, let \(E'=\Sphere(W')\) be another such hypersphere and choose
\(g\in SO(n+1)\) with \(gE=E'\).  For every \(L\in\K_p(E')\),
rotation invariance gives
\[
 \mu(L)=\mu(g^{-1}L)
 =\sum_{j=0}^{n-1}a_jV_j(g^{-1}L)
 =\sum_{j=0}^{n-1}a_jV_j(L).
\]
By uniqueness in dimension \(n-1\), these are also the coefficients
of the restriction to \(E'\).

Let \(K\in\K_p(\Sphere^n)\) be nonempty and lower-dimensional,
and choose a great hypersphere \(E'=\Sphere(W')\) containing it.
Choose \(u\in\Sphere^n\) such that \(\langle u,x\rangle>0\) for
every \(x\in K\), and let \(u'\) be the orthogonal projection of
\(u\) onto \(W'\).  Then
\[
 \langle u',x\rangle=\langle u,x\rangle>0\qquad(x\in K).
\]
Consequently, \(u'\neq0\), and \(u'/\|u'\|\) is the center of
an open hemisphere of \(E'\) containing \(K\).
Hence \(K\in\K_p(E')\), and the preceding expansion shows that
\[
 \eta:=\mu-\sum_{j=0}^{n-1}a_jV_j
\]
is simple.
Proposition~\ref{prop:continuous-simple-extension} extends \(\eta\)
to a continuous simple invariant valuation \(\bar\eta\) on
\(\K(\Sphere^n)\).  Theorem~\ref{prop:simple-classification} gives
\(\bar\eta=\bar\eta(\Sphere^n)V_n\).
Restriction to the proper domain completes the induction, with
\(a_n=\bar\eta(\Sphere^n)\).
\end{proof}

\begin{proof}[Proof of Theorem~\ref{thm:main}]
By Proposition~\ref{prop:proper-classification}, the restriction of
\(\mu\) to \(\K_p(\Sphere^n)\) has an expansion
\(\sum_{j=0}^{n}a_jV_j\).
Both \(\mu\) and this linear combination are valuation extensions
of the same proper-domain valuation.
Lemma~\ref{lem:algebraic-extension} therefore gives the expansion
on \(\K(\Sphere^n)\).
Evaluating on \(S_k\) yields \(a_k=\mu(S_k)\), proving the formula
and uniqueness.
\end{proof}

\begin{corollary}\label{cor:spherical-form}
For every \(n\geq1\), restriction induces an isomorphism from the space
of continuous \(SO(n+1)\)-invariant valuations on \(\K(\Sphere^n)\)
onto the corresponding space on \(\K_p(\Sphere^n)\).  The valuations
\(V_0,\ldots,V_n\) form a basis of the former space, their restrictions
form a basis of the latter, and both spaces have dimension \(n+1\).
Moreover, every valuation in either space is \(O(n+1)\)-invariant.
\end{corollary}

\begin{proof}
Lemma~\ref{lem:algebraic-extension} gives injectivity of the restriction map.
For surjectivity, Proposition~\ref{prop:proper-classification} expresses
every valuation on
\(\K_p(\Sphere^n)\) as a linear combination of the restrictions of
\(V_0,\ldots,V_n\), and the same linear combination on
\(\K(\Sphere^n)\) is a continuous invariant extension.  The basis and
dimension statements follow, and the \(O(n+1)\)-invariance of each
\(V_j\) yields the final assertion.
\end{proof}

\subsection{The conic classification}\label{subsec:conic-classification}

The restriction \(d\geq2\) is essential:
\(SO(1)\) is trivial, and an \(SO(1)\)-invariant conic valuation may
distinguish the two rays in \(\R\).

\begin{corollary}[Conic Hadwiger theorem]\label{cor:conic-form}
Let \(d\geq2\), and let \(\Phi:\C_d\to\R\) be a continuous,
not necessarily normalized, \(SO(d)\)-invariant conic valuation.  Then
\[
 \Phi(C)=\sum_{k=0}^{d}\Phi(L_k)v_k(C)
 \qquad(C\in\C_d).
\]
Here \(L_k\) is any \(k\)-dimensional linear subspace of \(\R^d\), with
\(L_0=\{0\}\).  Since \(SO(d)\) acts transitively on the
\(k\)-dimensional linear subspaces of \(\R^d\), the value \(\Phi(L_k)\)
is independent of this choice.
Moreover:
\begin{enumerate}
\renewcommand{\labelenumi}{\textup{(\roman{enumi})}}
\setlength{\itemsep}{2pt}
\setlength{\topsep}{3pt}
\item The valuations \(v_0,\ldots,v_d\) form a basis; hence the space has
dimension \(d+1\).
\item Every such conic valuation is \(O(d)\)-invariant.
\item The normalized subspace is
\(\operatorname{span}\{v_1,\ldots,v_d\}\) and has dimension \(d\).
\end{enumerate}
\end{corollary}

\begin{proof}
Put \(\Psi=\Phi-\Phi(\{0\})v_0\).
Then \(\Psi(\{0\})=0\), so
\(\mu(K)=\Psi(\widehat K)\) is a continuous \(SO(d)\)-invariant
spherical valuation by Lemma~\ref{lem:link-homeomorphism}.
Theorem~\ref{thm:main}, with \(n=d-1\), and
\(V_{k-1}(C\cap\Sphere^{d-1})=v_k(C)\)
give \(\Psi=\sum_{k=1}^d b_kv_k\), also at \(\{0\}\).
Evaluating on \(L_k\), \(k\geq1\), gives
\(b_k=\Psi(L_k)=\Phi(L_k)\), proving the formula.
The identities \(v_j(L_k)=\delta_{jk}\) give linear independence;
evaluation at \(\{0\}\) identifies the normalized subspace.
Orthogonal invariance follows from that of the \(v_j\).
\end{proof}

\section{Extension to non-proper sets}\label{sec:extension}

An extension theorem from \(\K_p(\Sphere^n)\) to
\(\K(\Sphere^n)\), including continuity, is stated in
\cite[Proposition~2.1.2]{Hack}.  We give independent proofs of the
algebraic extension and, for simple valuations, of its continuity.

\subsection{Algebraic extension}\label{subsec:algebraic-extension}

Fix an orthonormal basis \(e_1,\ldots,e_{n+1}\), write
\(x_j=\langle x,e_j\rangle\), and set
\[
 \mathcal I:=\{-1,0,1\}^{n+1}\setminus\{0\}.
\]
For \(\tau=(\tau_1,\ldots,\tau_{n+1})\in\mathcal I\), let
\(z(\tau):=\#\{j:\tau_j=0\}\) and define
\[
 Q_\tau:=\{x\in\Sphere^n:\tau_jx_j\geq0\ \text{if }\tau_j\neq0,
 \quad x_j=0\ \text{if }\tau_j=0\}.
\]
Each \(Q_\tau\) is closed and spherical convex.  It is also proper:
for every \(x\in Q_\tau\),
\[
 \left\langle x,\sum_{j=1}^{n+1}\tau_je_j\right\rangle
 =\sum_{\{j:\tau_j\neq0\}}|x_j|>0,
\]
because \(x\) is a unit vector and its remaining coordinates vanish.
Thus \(K\cap Q_\tau\in\K_p(\Sphere^n)\) for every
\(K\in\K(\Sphere^n)\), including the case of an empty intersection.

\begin{lemma}\label{lem:coordinate-cut-formula}
Let \(\rho\) be a valuation on either \(\K_p(\Sphere^n)\) or
\(\K(\Sphere^n)\).  For every \(A\) in its domain,
\begin{equation}
 \rho(A)=\sum_{\tau\in\mathcal I}
 (-1)^{z(\tau)}\rho(A\cap Q_\tau).
 \label{eq:coordinate-cut-formula}
\end{equation}
\end{lemma}

\begin{proof}
Set
\[
 H_j^\pm:=\{x\in\Sphere^n:\pm x_j\geq0\},
 \qquad H_j:=\{x\in\Sphere^n:x_j=0\}.
\]
For every \(A\) in the domain of \(\rho\), its intersections with
\(H_j^+,H_j^-,H_j\) remain in that domain, and the valuation identity
implies
\[
 \rho(A)=\rho(A\cap H_j^+)+\rho(A\cap H_j^-)-\rho(A\cap H_j).
\]
Applying this identity successively for \(j=1,\ldots,n+1\) gives
\eqref{eq:coordinate-cut-formula}: choosing the zero-coordinate term
at \(z(\tau)\) cuts gives the sign \((-1)^{z(\tau)}\), while the
all-zero choice has empty spherical intersection and contributes zero.
\end{proof}

\begin{proof}[Proof of Lemma~\ref{lem:algebraic-extension}]
Define
\begin{equation}
 \bar\mu(K)=\sum_{\tau\in\mathcal I}
 (-1)^{z(\tau)}\mu(K\cap Q_\tau).
 \label{eq:extension-formula}
\end{equation}
Each intersection is proper or empty.  Since intersection with
\(Q_\tau\) preserves unions and intersections, the sum is a valuation.
Lemma~\ref{lem:coordinate-cut-formula}, applied on the proper domain,
shows that it extends \(\mu\).  The same formula applied to any full-domain
extension determines its values uniquely, so the construction is
independent of the chosen basis.
If \(g\in SO(n+1)\), the valuation \(K\mapsto\bar\mu(gK)\)
also extends \(\mu\) whenever \(\mu\) is invariant; uniqueness gives
invariance of \(\bar\mu\).  Simplicity follows because every
intersection with a lower-dimensional \(K\) is empty or lower-dimensional.
\end{proof}

\subsection{Continuity for simple valuations}
\label{subsec:continuous-simple-extension}

We prove the continuity assertion in
Proposition~\ref{prop:continuous-simple-extension}.
The algebraic extension formula involves intersections with closed
orthants, which need not preserve Hausdorff convergence.
For example, if \(p\) lies on the boundary of a closed orthant \(Q\)
and \(p_r\to p\) from outside \(Q\), then
\(\{p_r\}\cap Q=\emptyset\) while \(\{p\}\cap Q=\{p\}\).
Simplicity removes the discontinuity in valuation values when the
limiting intersection has lower dimension.  The following two
lemmas make this precise.

\begin{lemma}\label{lem:stable-intersections}
Let \(A_r\to A\) in the Hausdorff topology on closed subsets of
\(\Sphere^n\) induced by \(d_s\), with \(\emptyset\) adjoined
as an isolated point, and let \(Q\subset\Sphere^n\) be closed.  If
\[
 A\cap\operatorname{int}Q
 \quad\text{is dense in}\quad
 A\cap Q,
\]
then
\[
 A_r\cap Q\longrightarrow A\cap Q
\]
in the same Hausdorff topology.
\end{lemma}

\begin{proof}
The empty cases follow from isolation of \(\emptyset\) and positive
distance between disjoint closed subsets of the sphere.  Otherwise, closedness of
\(Q\) implies that every limit of points in \(A_r\cap Q\) belongs
to \(A\cap Q\).  Every point of \(A\cap\operatorname{int}Q\)
is approximated by points of \(A_r\cap Q\).  Choose a finite net
for \(A\cap Q\) from this dense subset.  Approximating its finitely
many points gives the reverse Hausdorff bound uniformly, proving
convergence.
\end{proof}

\begin{lemma}\label{lem:simple-slice-continuity}
Let \(Q\in\K_p(\Sphere^n)\) be full-dimensional, and let
\(\mu:\K_p(\Sphere^n)\to\R\) be a continuous simple valuation.
Then the map
\[
 K\longmapsto\mu(K\cap Q),\qquad K\in\K(\Sphere^n),
\]
is a continuous valuation on \(\K(\Sphere^n)\).
\end{lemma}

\begin{proof}
Intersection with \(Q\) preserves unions and intersections, so the
map is a valuation.  Let \(K_r\to K\).
If \(K=\emptyset\), then \(K_r=\emptyset\) eventually.
If \(K\neq\emptyset\) and \(K\cap Q=\emptyset\), the disjoint compact
sets \(K\) and \(Q\) have positive distance, so \(K_r\cap Q=\emptyset\)
eventually.  It remains to consider \(K\cap Q\neq\emptyset\).
If \(K\cap Q\) is full-dimensional, choose
\(y\in\operatorname{int}(K\cap Q)\).  For every \(q\in K\cap Q\)
and \(0<t\leq1\), put
\[
 q_t=\frac{(1-t)q+ty}{\|(1-t)q+ty\|}.
\]
Since \(Q\) is proper, \(q\) and \(y\) lie in a common open
hemisphere, so the denominator is nonzero.  Convexity gives
\(q_t\in K\), and \(y\in\operatorname{int}Q\) implies
\(q_t\in\operatorname{int}Q\).  As \(q_t\to q\) when \(t\to0\),
\(K\cap\operatorname{int}Q\) is dense in \(K\cap Q\).
Lemma~\ref{lem:stable-intersections} and continuity of \(\mu\)
give convergence of the values.

If \(K\cap Q\) is lower-dimensional, every subsequence of nonempty
intersections \(K_r\cap Q\) has, by
Lemma~\ref{lem:spherical-hyperspace-compactness}, a further
subsequence converging to a nonempty \(A\subset K\cap Q\), since
the empty set is isolated.
Since \(A\subset Q\), it is proper and lower-dimensional.
Continuity and simplicity give convergence of the values to
\(\mu(A)=0=\mu(K\cap Q)\).  Empty intersections also have value
zero, so the whole sequence of values converges to zero.
\end{proof}

\begin{proof}[Proof of Proposition~\ref{prop:continuous-simple-extension}]
Let \(\bar\mu\) be the unique simple extension given by
Lemma~\ref{lem:algebraic-extension}; it is invariant whenever \(\mu\) is.
It remains to prove continuity.
If \(z(\tau)>0\), then \(Q_\tau\) lies in a great hypersphere, so
simplicity gives \(\mu(K\cap Q_\tau)=0\) for every
\(K\in\K(\Sphere^n)\).  Thus \eqref{eq:extension-formula} reduces to
\[
 \bar\mu(K)=\sum_{\tau\in\{-1,1\}^{n+1}}\mu(K\cap Q_\tau).
\]
Each \(Q_\tau\) is full-dimensional and proper.  By
Lemma~\ref{lem:simple-slice-continuity}, every summand is continuous
on \(\K(\Sphere^n)\), and hence so is \(\bar\mu\).
\end{proof}

\section*{Acknowledgements}

This work is supported by the National Natural Science Foundation of
China (Grant No.~12571350) and the Guangdong Basic and Applied Basic
Research Foundation (Grant No.~2025A1515010457).

\section*{Declaration on the Use of AI}

OpenAI Codex (GPT-5.6-sol) provided substantial assistance in developing and revising
parts of the mathematical arguments in this manuscript.  It also
assisted with identifying gaps and organizing and editing the text.
The authors have independently verified all statements, proofs, and
references.  They made all final mathematical and editorial decisions
and take full responsibility for the correctness and integrity of the
manuscript.

\bigskip
\noindent\textsc{School of Mathematics, Hunan University, Changsha 410082,
Hunan, P. R. China}

\noindent\textit{E-mail addresses:}
\href{mailto:wangsuijie@hnu.edu.cn}{\texttt{wangsuijie@hnu.edu.cn}} (Suijie~Wang),\\
\href{mailto:2072704745@hnu.edu.cn}{\texttt{2072704745@hnu.edu.cn}} (Shengguo~Wu).

\end{document}